\documentclass[graybox]{svmult}

\usepackage{type1cm}        
\usepackage{makeidx}         
\usepackage{graphicx}        
\usepackage{multicol}        
\usepackage[bottom]{footmisc}

\usepackage{newtxtext}       %
\usepackage[varvw]{newtxmath}       

\makeindex             

\begin{document}

\title*{A perturbation result for a Neumann problem in a periodic domain}
\author{Matteo Dalla Riva, Paolo Luzzini, and Paolo Musolino}

\institute{Matteo Dalla Riva \at Dipartimento di Ingegneria, Universit\`a degli Studi di Palermo, Viale delle Scienze, Ed. 8, 90128 Palermo, Italy, \email{matteo.dallariva@unipa.it}
\and Paolo Luzzini \at Dipartimento di Matematica `Tullio Levi-Civita', Universit\`a degli Studi di Padova, Via Trieste 63, Padova 35121, Italy, \email{pluzzini@math.unipd.it}
\and Paolo Musolino  \at Dipartimento di Scienze Molecolari e Nanosistemi, Universit\`a Ca' Foscari Venezia, via Torino 155, 30172 Venezia Mestre, Italy, \email{paolo.musolino@unive.it}}

%
%
\maketitle

\abstract{We consider a Neumann problem for the Laplace equation in a periodic domain. We prove  that the solution  depends real analytically 
on the shape of the domain, on the periodicity parameters, on the Neumann datum, and on its boundary integral.}

\section{Introduction}

The aim of this paper is to prove the analytic dependence of the solution of a periodic Neumann problem for the Laplace equation, upon joint perturbation of the domain, the periodicity parameters, the Neumann datum, and  its integral on the boundary.  The domain is obtained as the union of congruent copies of a periodicity cell of edges of length 
$q_{11},\ldots,q_{nn}$ with a hole whose shape is the image of a reference domain through a diffeomorphism $\phi$. As Neumann datum we take the projection of a function $g$, defined on the boundary of the reference domain and suitably rescaled,  on the space of functions with zero integral on the boundary. As it happens for non-periodic Neumann problems, in order to identify one solution, we impose that the integral of the solution on the boundary is equal to a given real constant $k$. By means of a periodic version of potential theory,  we prove that the solution of the problem depends 
real analytically 
on the `periodicity-domain-Neuman datum-integral' quadruple $((q_{11},\ldots,q_{nn}), \phi,g,k)$.

Many authors have investigated the behavior of the solutions to boundary value problems upon domain perturbations. 
 We mention,  e.g., Henry \cite{He82} and Sokolowski and Zol\'esio \cite{SoZo92} for elliptic domain perturbation problems. Lanza de Cristoforis  \cite{La05, La07} has exploited potential theory in order to prove that the solutions of boundary value problems for the Laplace and Poisson equations depend real analytically upon domain perturbation. Moreover, analyticity results for domain perturbation problems for eigenvalues have been obtained for example for the Laplace equation by Lanza de Cristoforis and Lamberti \cite{LaLa04}, for the biharmonic operator by Buoso and Provenzano \cite{BuPr15}, and for the Maxwell's equations by Lamberti and Zaccaron \cite{LaZa21}.

In order to introduce our problem, we fix once for all a natural number
 \[
 n \in \mathbb{N} \setminus \{0,1\}\, 
 \] 
that represents the dimension of the space.
If $(q_{11},\ldots,q_{nn}) \in \mathopen]0,+\infty[^n$ we define a periodicity cell $Q$ and a matrix $q \in {\mathbb{D}}_{n}^{+}({\mathbb{R}})$ as
\[
Q \equiv \prod_{j=1}^n \mathopen]0,q_{jj}[, \quad
q \equiv 
 \begin{pmatrix}
  q_{11} & 0 & \cdots & 0 \\
  0 & q_{22} & \cdots & 0 \\
  \vdots  & \vdots  & \ddots & \vdots  \\
  0 & 0 & \cdots & q_{nn} 
 \end{pmatrix},
\]
where ${\mathbb{D}}_{n}({\mathbb{R}})$ is the space of $n\times n$ diagonal matrices with real entries and ${\mathbb{D}}_{n}^{+}({\mathbb{R}})$ is the set of elements of 
${\mathbb{D}}_{n}({\mathbb{R}})$ with diagonal entries in $]0,+\infty[$. Here we note that we can identify ${\mathbb{D}}_{n}^{+}({\mathbb{R}})$ and $]0,+\infty[^n$.
We denote by $|Q|_n$  the $n$-dimensional measure of the cell $Q$, by $\nu_Q$ the outward unit normal to $\partial Q$, where it exists, and by $q^{-1}$ the inverse matrix of $q$. We find convenient to set
\[
\widetilde{Q}\equiv \mathopen]0,1[^n\, , \qquad \tilde{q}  \equiv I_n \,,
\]
where $I_n$ denotes the identity $n\times n$ matrix. Then we introduce the reference domain: we take
\begin{equation}\label{Omega_def}
\begin{split}
&\text{$\alpha \in \mathopen]0,1[$ and a bounded open connected subset $\Omega$ of $\mathbb{R}^{n}$}
\\
&\text{of  class  $C^{1,\alpha}$ such that $\mathbb{R}^{n}\setminus\overline{\Omega}$ is  connected}\, ,
\end{split}
\end{equation}
where the symbol `$\overline{\cdot}$' denotes the closure of a set.
For the definition of sets and functions of the Schauder class $C^{1,\alpha}$ we refer,  e.g., to Gilbarg and
Trudinger~\cite{GiTr83}.  In order to model our variable domain 
we consider a class of diffeomorphisms 
${\mathcal{A}}_{  \partial\Omega  }^{\widetilde{Q}}$ from $\partial\Omega$ into  their images contained in $\widetilde{Q}$ (see \eqref{eq:defA} below). 
By the Jordan-Leray separation theorem, if $\phi\in {\mathcal{A}}_{\partial\Omega}^{\widetilde{Q}}$,  
 the set
${\mathbb{R}}^{n}\setminus\phi (\partial\Omega)$ has exactly two open 
connected components (see, e.g., Deimling \cite[Thm.  5.2, p. 26]{De85}). We denote by ${\mathbb{I}}[\phi]$  the 
bounded  open connected component of 
${\mathbb{R}}^{n}\setminus\phi (\partial\Omega)$.
Since $\phi (\partial\Omega)\subseteq \widetilde{Q}$, a topological argument shows that 
$\widetilde{Q}\setminus\overline{\mathbb{I}[\phi]}$
is also connected (cf., e.g., \cite[Theorem A.10]{DaLaMu21}). 
We are now in the position to introduce the following two periodic domains (see Figure \ref{fig1}):
\begin{equation*}
\mathbb{S}_{q }[q \mathbb{I}[\phi]] \equiv\bigcup_{z\in {\mathbb{Z}}^{n}}\left(
q z+q {\mathbb{I}}[\phi] 
\right), \qquad
\mathbb{S}_{q }[q \mathbb{I}[\phi]]^- \equiv
{\mathbb{R}}^{n}\setminus\overline{\mathbb{S}_{q }[q \mathbb{I}[\phi]]} \,.
\end{equation*}
The set $\mathbb{S}_{q }[q \mathbb{I}[\phi]]^-$ will be the one where we shall set our Neumann problem. 
 Clearly, a perturbation of $q$ produces a modification of the whole periodicity structure of  $\mathbb{S}_{q }[q \mathbb{I}[\phi]]^-$, while a perturbation of $\phi$ induces a change in the shape of the holes $\mathbb{S}_{q }[q \mathbb{I}[\phi]]$.

\begin{figure}
\sidecaption
\includegraphics[scale=.18]{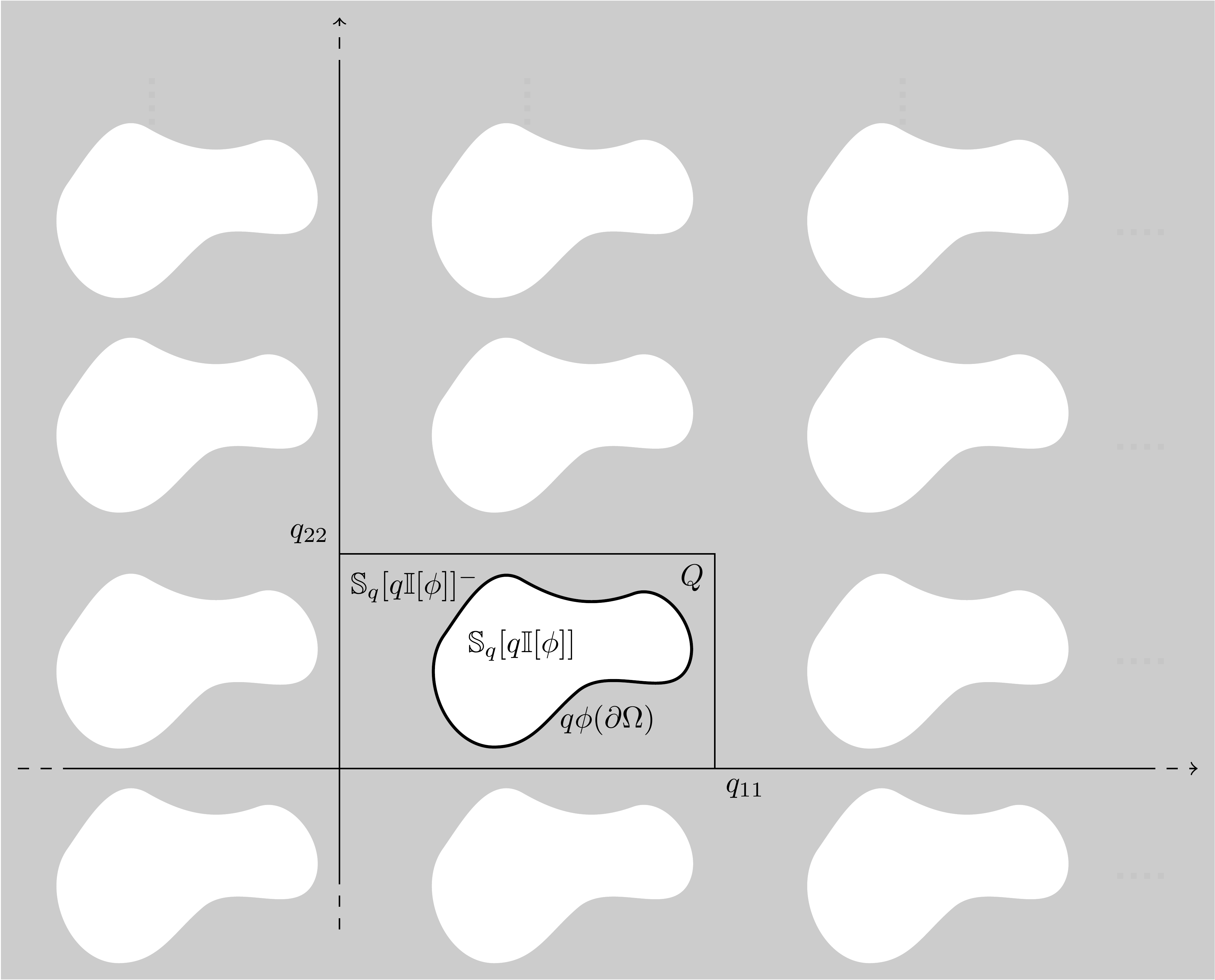}
%
%
\caption{The sets $\mathbb{S}_{q}[q\mathbb{I}[\phi]]^-$ (in gray),  
$\mathbb{S}_{q}[q\mathbb{I}[\phi]]$ (in white), and  $q\phi(\partial\Omega)$ (in black) in case $n=2$.}
\label{fig1}       
\end{figure}

If  $q \in \mathbb{D}^+_n(\mathbb{R})$, $\phi \in C^{1,\alpha}(\partial\Omega,\mathbb{R}^n) \cap {\mathcal{A}}_{\partial\Omega}^{\widetilde{Q}}$, $g \in C^{0,\alpha}(\partial \Omega)$ and $k \in \mathbb{R}$, we consider the following periodic Neumann problem for the Laplace equation:
\begin{equation}
\label{bvp}
\left\{
\begin{array}{ll}
\Delta u=0 & {\mathrm{in}}\ \mathbb{S}_{q}[q\mathbb{I}[\phi]]^{-}\,,
\\
u(x+qz)=u(x)& \forall x \in \overline{\mathbb{S}_{q}[q\mathbb{I}[\phi]]^{-}}\, , \forall z \in\mathbb{Z}^n\,,\\
\frac{\partial}{\partial \nu_{q\mathbb{I}[\phi]}}u(x)=g\big( \phi^{(-1)}(q^{-1}x)\big)&\\
\qquad \qquad-\frac{1}{\int_{\partial q\mathbb{I}[\phi]}\, d\sigma}\int_{\partial q\mathbb{I}[\phi]} g\big( \phi^{(-1)}(q^{-1}y)\big) \, d\sigma_y& \forall x\in  \partial q\mathbb{I}[\phi] \,, \\
\int_{\partial q\mathbb{I}[\phi]}u\, d\sigma=k\, .&  \\
\end{array}
\right.
\end{equation}
We note that the function
\[
g\big( \phi^{(-1)}(q^{-1}\cdot)\big)-\frac{1}{\int_{\partial q\mathbb{I}[\phi]}\, d\sigma}\int_{\partial q\mathbb{I}[\phi]} g\big( \phi^{(-1)}(q^{-1}y)\big) \, d\sigma_y
\]
clearly belongs to the space
\[
C^{0,\alpha}(\partial q\mathbb{I}[\phi])_0 \equiv \Big\{\mu \in C^{0,\alpha}(\partial q\mathbb{I}[\phi]) \colon \int_{\partial q\mathbb{I}[\phi]}\mu\, d\sigma=0\Big\}\, .
\]
As a consequence,  the solution of problem \eqref{bvp} in the space $C^{1,\alpha}_{q}(\overline{\mathbb{S}_{q}[q\mathbb{I}[\phi]]^{-}})$ of $q$-periodic functions in $\overline{\mathbb{S}_{q}[q\mathbb{I}[\phi]]^-}$ of class $C^{1,\alpha}$ exists and  is unique and we denote it by $u[q,\phi,g,k]$ (see \cite[Thm.~12.23]{DaLaMu21}).  Our aim is to prove that $u[q,\phi,g,k]$ depends, in a sense that we will clarify, analytically on $(q,\phi,g,k)$  (see Theorem \ref{mainthm}). 
Our work originates from Lanza de Cristoforis  \cite{La05, La07} on the real analytic dependence 
of the solution of the Dirichlet problem for the Laplace and Poisson equations upon domain perturbations. Moreover, this paper can be seen as the Neumann counterpart of \cite{LuMu22}, where the authors have proved analyticity properties for the solution of a periodic Dirichlet problem. An analysis similar to the one of the present paper was also carried out for periodic problems related to physical quantities arising in fluid mechanics and in material science (see \cite{DaLuMuPu21, LuMu20, LuMuPu19}).

\section{Preliminary results}

In order to consider shape perturbations, we introduce a class of diffeomorphisms. Let $\Omega$ be as in \eqref{Omega_def}. Let  $\mathcal{A}_{\partial \Omega}$ be the set of functions of class $C^1(\partial\Omega, \mathbb{R}^{n})$ which are injective and whose differential is injective  at all points  of $\partial\Omega$. The set $\mathcal{A}_{\partial \Omega}$ is well-known to be open 
in $C^1(\partial\Omega, \mathbb{R}^{n})$
(see,  e.g., Lanza de Cristoforis and Rossi \cite[Lem. 2.5, p. 143]{LaRo04}).
Then we  set
\begin{equation}
\label{eq:defA}
{\mathcal{A}}_{\partial\Omega}^{\widetilde{Q}} \equiv \Big\{\phi \in\mathcal{A}_{\partial \Omega} : \phi(\partial\Omega) \subseteq \widetilde{Q}\Big\}.
\end{equation}

In order to analyze our boundary value problem, we are going to exploit periodic layer potentials. To define these operators, it is enough to replace the fundamental solution of the Laplace operator by a $q$-periodic tempered distribution $S_{q,n}$ such that $\Delta S_{q,n}=\sum_{z\in {\mathbb{Z}}^{n}}\delta_{qz}-\frac{1}{|Q|_n}$,
where $\delta_{qz}$ is the Dirac measure with mass in $qz$
(see e.g., \cite[Chapter 12]{DaLaMu21}).
We can take
\begin{equation*}
S_{q,n}(x)=-\sum_{ z\in {\mathbb{Z}}^{n}\setminus\{0\} }
\frac{1}{     |Q|_n4\pi^{2}|q^{-1}z|^{2}   }e^{2\pi i (q^{-1}z)\cdot x}
\end{equation*}
in the sense of distributions in $ {\mathbb{R}}^{n}$ 
(see e.g.,  Ammari and Kang~\cite[p.~53]{AmKa07}, \cite[\S 12.1]{DaLaMu21}). 
Moreover,  $S_{q,n}$ is even, real analytic in ${\mathbb{R}}^{n}\setminus q{\mathbb{Z}}^{n}$, and  locally integrable in ${\mathbb{R}}^{n}$
 (see e.g., \cite[Thm.~12.4]{DaLaMu21}).  We now introduce the periodic single layer potential.
Let $\Omega_Q$ be a bounded open  subset of ${\mathbb{R}}^{n}$ of class $C^{1,\alpha}$ for some $\alpha\in\mathopen]0,1[$ such that $\overline{\Omega_Q}\subseteq Q$. We define the following two periodic domains:
\begin{equation*}
\mathbb{S}_{q }[\Omega_Q] \equiv\bigcup_{z\in {\mathbb{Z}}^{n}}\left(
q z+\Omega_Q 
\right), \qquad
\mathbb{S}_{q }[\Omega_Q]^- \equiv
{\mathbb{R}}^{n}\setminus\overline{\mathbb{S}_{q }[\Omega_Q]} \,
\end{equation*}
and we set
\[
v_q[\partial\Omega_Q, \mu](x)\equiv
\int_{\partial\Omega_Q} S_{q,n}(x-y)\mu(y)\,d\sigma_{y}\qquad\forall x\in {\mathbb{R}}^{n}\,
\]
and
\[
W_q^\ast[\partial\Omega_Q, \mu](x)\equiv
\int_{\partial\Omega_Q}
 \nu_{\Omega_Q}(x) \cdot 
DS_{q,n}(x-y)\mu(y)\,d\sigma_{y}\qquad\forall x\in  \partial\Omega_Q
\]
for all $\mu\in L^{2}(\partial\Omega_Q)$.  The symbol $\nu_{\Omega_Q}$ denotes the outward unit normal 
field to $\partial\Omega_Q$,  $d\sigma$ denotes the area element on $\partial\Omega_Q$ and $DS_{q,n}$ denotes the gradient of $S_{q,n}$.
The function $v_q[\partial\Omega_Q, \mu]$ is called the $q$-periodic single layer potential.
  Now let $\mu\in C^{0,\alpha}(\partial\Omega_Q)$. As is well known,  $v^+_q[\partial\Omega_Q, \mu]\equiv v_q[\partial\Omega_Q, \mu]_{|\overline{\mathbb{S}_q[\Omega_Q]}}$ belongs to $C_{q}^{1,\alpha}(\overline{\mathbb{S}_{q}[\Omega_Q]})$ and $v^-_q[\partial\Omega_Q, \mu]\equiv v_q[\partial\Omega_Q, \mu]_{|\overline{\mathbb{S}_q[\Omega_Q]^-}}$ belongs to $C_{q}^{1,\alpha}(\overline{\mathbb{S}_{q}[\Omega_Q]^-})$ (see \cite[Thm.~12.8]{DaLaMu21}). Moreover, the following jump formula holds: 
\[
\frac{\partial }{\partial \nu_{\Omega_Q}}v_q^\pm[\partial\Omega_Q, \mu] = \mp \frac{1}{2}\mu + W_q^\ast[\partial\Omega_Q, \mu] \qquad \mbox{ on } \partial\Omega_Q.
\]
For a proof of the above formula we refer to \cite[Thm.~12.11]{DaLaMu21}.

 Since our approach will be based on integral operators, we need to understand how  integrals behave when we perturb the domain of  integration. Moreover, we need also to understand the regularity of the normal vector upon domain perturbations. For such reasons, we collect those results in the lemma below   (for a proof, see Lanza de Cristoforis and Rossi \cite[p.~166]{LaRo04}).
\begin{lemma}\label{rajacon}
 Let $\alpha$, $\Omega$ be as in \eqref{Omega_def}.  Then the following statements hold.
\begin{itemize}
\item[(i)] For each $\psi \in C^{1,\alpha}(\partial\Omega, \mathbb{R}^{n})\cap\mathcal{A}_{\partial \Omega}$, there exists a unique  
$\tilde \sigma[\psi] \in C^{0,\alpha}(\partial\Omega)$ such that $\tilde \sigma[\psi] > 0$ and 
\[ 
  \int_{\psi(\partial\Omega)}\omega(s)\,d\sigma_s=  \int_{\partial\Omega}\omega \circ \psi(y)\tilde\sigma[\psi](y)\,d\sigma_y, \qquad \forall \omega \in L^1(\psi(\partial\Omega)).
\]
Moreover, the map $\tilde \sigma[\cdot]$ from $C^{1,\alpha}(\partial\Omega, \mathbb{R}^n)\cap\mathcal{A}_{\partial \Omega}  $ to $ C^{0,\alpha}(\partial\Omega)$ is real analytic.
\item[(ii)]  The map from $C^{1,\alpha}(\partial\Omega, \mathbb{R}^n)\cap\mathcal{A}_{\partial \Omega} $ to $ C^{0,\alpha}(\partial\Omega, \mathbb{R}^{n})$ which takes $\psi$ to $\nu_{\mathbb{I}[\psi]} \circ \psi$ is real analytic.
\end{itemize}
\end{lemma}

\section{Analyticity of the solution}

Our first goal is to transform problem \eqref{bvp} into an integral equation. In order to analyze the solvability of the obtained integral equation, we need the following lemma.

\begin{lemma}\label{leminteq}
Let $q \in  {\mathbb{D}}_{n}^{+}({\mathbb{R}})$. Let $\alpha$, $\Omega$ be as in \eqref{Omega_def}. 
	  Let $\phi\in C^{1,\alpha}(\partial\Omega, \mathbb{R}^n) \cap {\mathcal{A}}_{\partial\Omega}^{\widetilde{Q}}$.
Let $N$ be the map from $C^{0,\alpha}(\partial q\mathbb{I}[\phi])$  to itself, defined by
\[
N[\mu] \equiv \frac{1}{2}\mu+W_{q}^\ast[\partial q\mathbb{I}[\phi], \mu] \qquad \forall \mu \in C^{0,\alpha}(\partial q\mathbb{I}[\phi]).
\]
Then $N$ is a linear homeomorphism from 
$C^{0,\alpha}(\partial q\mathbb{I}[\phi])$ to itself. Moreover, $N$ restricts to a linear homeomorphism from 
$C^{0,\alpha}(\partial q\mathbb{I}[\phi])_0$ to itself.
\end{lemma}
\begin{proof}
By \cite[Thm.~12.20]{DaLaMu21}, we deduce that $N$ is a linear homeomorphism from 
$C^{0,\alpha}(\partial q\mathbb{I}[\phi])$ to itself. By \cite[Prop.~12.15]{DaLaMu21}, we have that $\frac{1}{2}\mu+W_{q}^\ast[\partial q\mathbb{I}[\phi], \mu]$ belongs to $C^{0,\alpha}(\partial q\mathbb{I}[\phi])_0$ if and only if $\mu$ belongs to $C^{0,\alpha}(\partial q\mathbb{I}[\phi])_0$. As a consequence, we also have that $N$ restricts to a linear homeomorphism from 
$C^{0,\alpha}(\partial q\mathbb{I}[\phi])_0$ to itself.
\end{proof}

Then, in the following proposition, we show how to convert the Neumann problem into an equivalent integral equation.

\begin{proposition}\label{propAUX}
		  Let $\alpha$, $\Omega$ be as in \eqref{Omega_def}. Let $q \in  {\mathbb{D}}_{n}^{+}({\mathbb{R}})$. 
		Let $\phi\in C^{1,\alpha}(\partial\Omega, \mathbb{R}^n) \cap {\mathcal{A}}_{\partial\Omega}^{\widetilde{Q}}$.
	Let $g \in C^{0,\alpha}(\partial \Omega)$. Let $k \in \mathbb{R}$. Then the boundary value problem 
	\begin{equation}\label{np}
	\left\{
\begin{array}{ll}
\Delta u=0 & {\mathrm{in}}\ \mathbb{S}_{q}[q\mathbb{I}[\phi]]^{-}\,,
\\
u(x+qz)=u(x)& \forall x \in \overline{\mathbb{S}_{q}[q\mathbb{I}[\phi]]^{-}}\, , \forall z \in\mathbb{Z}^n\,,\\
\frac{\partial}{\partial \nu_{q\mathbb{I}[\phi]}}u(x)=g\big( \phi^{(-1)}(q^{-1}x)\big)&\\
\qquad \qquad-\frac{1}{\int_{\partial q\mathbb{I}[\phi]}\, d\sigma}\int_{\partial q\mathbb{I}[\phi]} g\big( \phi^{(-1)}(q^{-1}y)\big) \, d\sigma_y& \forall x\in  \partial q\mathbb{I}[\phi] \,, \\
\int_{\partial q\mathbb{I}[\phi]}u\, d\sigma=k\, &  \\
\end{array}
\right.
	\end{equation}
	has a unique solution $u[q,\phi,g,k]$ in  $C_{q}^{1,\alpha}(\overline{\mathbb{S}_{q}[q\mathbb{I}[\phi]]^-})$. Moreover,
	\begin{equation}\label{solAdd}
	\begin{split}
	u[q,\phi,&g,k](x)=v_{q}^-[\partial q\mathbb{I}[\phi], \mu](x)\\
	&+\frac{1}{\int_{\partial q\mathbb{I}[\phi]}\, d\sigma} \Bigg(k-\int_{\partial q\mathbb{I}[\phi]}v_{q}^-[\partial q\mathbb{I}[\phi], \mu]\, d\sigma\Bigg) \qquad \forall x\in \overline{\mathbb{S}_{q}[q\mathbb{I}[\phi]]^-},
	\end{split}
	\end{equation}
	where $\mu$ is the unique solution in $C^{0,\alpha}(\partial q\mathbb{I}[\phi])_0$ of the integral equation
	\begin{equation}\label{intEq}
	\begin{split}
	\frac{1}{2}\mu(x)&+W_{q}^\ast[\partial q\mathbb{I}[\phi], \mu](x)=g\big( \phi^{(-1)}(q^{-1}x)\big)\\
	&-\frac{1}{\int_{\partial q\mathbb{I}[\phi]}\, d\sigma}\int_{\partial q\mathbb{I}[\phi]} g\big( \phi^{(-1)}(q^{-1}y)\big) \, d\sigma_y\qquad \forall x\in \partial q\mathbb{I}[\phi]\, .
	\end{split}
	\end{equation}
\end{proposition} 
\begin{proof}
By \cite[Thm.~12.23]{DaLaMu21} we know that problem \eqref{np} has a unique solution. Moreover, by Lemma \ref{leminteq}, equation \eqref{intEq} has a unique solution $\mu$ which belongs to $C^{0,\alpha}(\partial q\mathbb{I}[\phi])_0$. Then by the properties of the periodic single layer potential, we deduce that the right hand side of \eqref{solAdd} solves problem \eqref{np}.
\end{proof}

In Proposition \ref{propAUX}, we have seen an integral equation on $\partial q\mathbb{I}[\phi]$ equivalent to problem \eqref{bvp}. However, if we want to study the dependence of the solution of the integral equation on the parameters $(q,\phi,g,k)$, it may be convenient to transform the equation on the $(q,\phi)$-dependent set $\partial q\mathbb{I}[\phi]$ into an equation on a fixed domain. We do so in the lemma below.

\begin{lemma}\label{lemSyst}
	  Let $\alpha$, $\Omega$ be as in \eqref{Omega_def}. Let $q \in  {\mathbb{D}}_{n}^{+}({\mathbb{R}})$. 
	Let $\phi  \in C^{1,\alpha}(\partial\Omega,\mathbb{R}^n) \cap {\mathcal{A}}_{\partial\Omega}^{\widetilde{Q}}$. Let $g \in C^{0,\alpha}(\partial \Omega)$. Then the function $\theta\in C^{0,\alpha}(\partial\Omega)$ solves the equation
	\begin{equation}\label{intEq1}
	\begin{split}
	\frac{1}{2}\theta(t) +\int_{q\phi(\partial\Omega)}\,& \nu_{q\mathbb{I}[\phi]}(q\phi(t)) \cdot  DS_{q,n}(q\phi(t)-y)\theta\big( \phi^{(-1)}(q^{-1}y)\big)d\sigma_y\\&=g(t)-\frac{1}{\int_{ \partial \Omega }\tilde{\sigma}[q\phi]\, d\sigma}\int_{\partial \Omega } g\tilde{\sigma}[q\phi] \, d\sigma
	\qquad \forall t\in \partial\Omega\, ,
	\end{split}
	\end{equation}
	if and only if the function $\mu\in C^{0,\alpha}(\partial q\mathbb{I}[\phi])$, with $\mu$ delivered by 
	 \begin{equation}\label{mudef}
	 \mu(x)=\theta\big( \phi^{(-1)}(q^{-1}x)\big)	\qquad \forall x\in \partial q\mathbb{I}[\phi],
	 \end{equation}
	 solves the equation 
\[
\begin{split}
	&\frac{1}{2}\mu(x)+W_{q}^\ast[\partial q\mathbb{I}[\phi], \mu](x)\\
	&=g\big( \phi^{(-1)}(q^{-1}x)\big)  -\frac{1}{\int_{\partial q\mathbb{I}[\phi]}\, d\sigma}\int_{\partial q\mathbb{I}[\phi]} g\big( \phi^{(-1)}(q^{-1}y)\big) \, d\sigma_y \quad\forall x\in \partial q\mathbb{I}[\phi]\, .
\end{split}
\]
	 Moreover, equation \eqref{intEq1} has a unique solution $\theta$ in $C^{0,\alpha}(\partial\Omega)$ and the function $\mu$ delivered by \eqref{mudef} belongs to $C^{0,\alpha}(\partial q\mathbb{I}[\phi])_0$. 
\end{lemma}
 \begin{proof}
It is a direct consequence of the theorem of change of variable in integrals, of Lemma \ref{leminteq}, and of the obvious equality
\[
\int_{\partial q\mathbb{I}[\phi]} \Bigg(g\big( \phi^{(-1)}(q^{-1}x)\big)  -\frac{1}{\int_{\partial q\mathbb{I}[\phi]}\, d\sigma}\int_{\partial q\mathbb{I}[\phi]} g\big( \phi^{(-1)}(q^{-1}y)\big) \, d\sigma_y\Bigg) d\sigma_x=0\, ,
\]
which implies that 
\[
g\big( \phi^{(-1)}(q^{-1}\cdot)\big)  -\frac{1}{\int_{\partial q\mathbb{I}[\phi]}\, d\sigma}\int_{\partial q\mathbb{I}[\phi]} g\big( \phi^{(-1)}(q^{-1}y)\big) \, d\sigma_y
\]
is in $C^{0,\alpha}(\partial q\mathbb{I}[\phi])_0$. \end{proof}
Our next goal is to study the dependence of the solution of the integral equation \eqref{intEq1} upon $(q,\phi,g)$. We wish to apply the implicit function theorem in Banach spaces. Therefore, having in mind equation \eqref{intEq1},  we  introduce the map $\Lambda$ from $ {\mathbb{D}}_{n}^{+}({\mathbb{R}})\times \left(C^{1,\alpha}(\partial\Omega, \mathbb{R}^n) \cap {\mathcal{A}}_{\partial\Omega}^{\widetilde{Q}}\right)\times \big(C^{0,\alpha}(\partial\Omega)\big)^2$ to $C^{0,\alpha}(\partial\Omega)$   by setting
 \[
 \begin{split}
 \Lambda&[q,\phi,g,\theta](t) \equiv 
 \frac{1}{2}\theta(t)\\ 
 & +\int_{q\phi(\partial\Omega)}\,\nu_{q\mathbb{I}[\phi]}(q\phi(t)) \cdot DS_{q,n}(q\phi(t)-y)\theta\big( \phi^{(-1)}(q^{-1}y)\big)d\sigma_y\\
 &-g(t)+\frac{1}{\int_{ \partial \Omega }\tilde{\sigma}[q\phi]\, d\sigma}\int_{\partial \Omega } g\tilde{\sigma}[q\phi] \, d\sigma\quad \forall t\in \partial\Omega,
\end{split}
\]
 for all $(q,\phi,g,\theta)\in{\mathbb{D}}_{n}^{+}({\mathbb{R}})\times\left(C^{1,\alpha}(\partial\Omega, \mathbb{R}^n) \cap {\mathcal{A}}_{\partial\Omega}^{\widetilde{Q}}\right)\times \big(C^{0,\alpha}(\partial\Omega)\big)^2$. 

We are now ready to apply the implicit function theorem for real analytic maps in Banach spaces to equation $\Lambda[q,\phi,g,\theta]=0$ and prove that the solution $\theta$ depends analytically on $(q,\phi,g)$.

\begin{proposition}\label{taxi}
	 Let $\alpha$, $\Omega$ be as in \eqref{Omega_def}. 
	  Then the following statements hold. 
\begin{itemize} 
\item[(i)] $\Lambda$ is real analytic.
\item[(ii)] For each $(q,\phi,g) \in {\mathbb{D}}_{n}^{+}({\mathbb{R}})\times\left(C^{1,\alpha}(\partial\Omega, \mathbb{R}^n) \cap {\mathcal{A}}_{\partial\Omega}^{\widetilde{Q}}\right) \times C^{0,\alpha}(\partial\Omega)$, there exists a unique 
$\theta$  in $C^{0,\alpha}(\partial\Omega)$ such that 
\[
\Lambda[q,\phi,g,\theta]=0 	\qquad \mbox{ on } \partial\Omega,
\]
 and we denote such a function by $\theta[q,\phi,g]$.
\item[(iii)] The map  $\theta[\cdot,\cdot,\cdot]$ from ${\mathbb{D}}_{n}^{+}({\mathbb{R}})\times\left(C^{1,\alpha}(\partial\Omega, \mathbb{R}^n) \cap {\mathcal{A}}_{\partial\Omega}^{\widetilde{Q}}\right)\times C^{0,\alpha}(\partial\Omega)$ to  $C^{0,\alpha}(\partial\Omega)$ that takes $(q,\phi,g)$ to $\theta[q,\phi,g]$ is real analytic.		
	\end{itemize}
\end{proposition}
\begin{proof}
By \cite[Thm. 3.2 (ii)]{LuMuPu20}, Lemma \ref{rajacon}, and standard calculus in Banach spaces, we deduce the validity of statement (i). Statement (ii) follows by Lemmas \ref{leminteq} and  \ref{lemSyst}. In order to prove (iii), since  the analyticity is a local property,  it suffices to fix $(q_0,\phi_0,g_0)$ in ${\mathbb{D}}_{n}^{+}({\mathbb{R}})\times\left(C^{1,\alpha}(\partial\Omega, \mathbb{R}^n) \cap {\mathcal{A}}_{\partial\Omega}^{\widetilde{Q}}\right)\times C^{0,\alpha}(\partial\Omega)$ and to show that $\theta[\cdot,\cdot,\cdot]$ is real analytic in a 
 neighborhood of $(q_0,\phi_0,g_0)$ in the product space ${\mathbb{D}}_{n}^{+}({\mathbb{R}})\times\left(C^{1,\alpha}(\partial\Omega, \mathbb{R}^n) \cap {\mathcal{A}}_{\partial\Omega}^{\widetilde{Q}}\right)\times C^{0,\alpha}(\partial\Omega)$.  By standard calculus in normed spaces, the partial 
 differential $\partial_{\theta}\Lambda[q_0,\phi_0,g_0,\theta[q_0,\phi_0,g_0]]$  of $\Lambda$ at $(q_0,\phi_0,g_0,\theta[q_0,\phi_0,g_0])$ with respect to the variable $\theta$ is delivered by
 \begin{align*}
\partial_{\theta}&\Lambda[q_0,\phi_0,g_0,\theta[q_0,\phi_0,g_0]](\psi)(t) &\\
=& \frac{1}{2}\psi(t) +\int_{q_0\phi_0(\partial\Omega)}\,\nu_{q_0\mathbb{I}[\phi_0]}(q_0\phi_0(t)) \cdot DS_{q_0,n}(q_0\phi_0(t)-y)\psi\big( \phi_0^{(-1)}(q_0^{-1}y)\big)d\sigma_y\\
 &  \hspace{9cm} \forall t\in \partial\Omega,
\end{align*}
for all $\psi \in C^{0,\alpha}(\partial\Omega)$. Lemma \ref{leminteq} together with a change of variable implies that $\partial_{\theta}\Lambda[q_0,\phi_0,g_0,\theta[q_0,\phi_0,g_0]]$ 
is a linear homeomorphism from $C^{0,\alpha}(\partial\Omega)$ onto  $C^{0,\alpha}(\partial\Omega)$. 
Finally, by the implicit function theorem for real analytic maps in Banach spaces
(see,  e.g., Deimling \cite[Thm. 15.3]{De85}) 
we deduce that $\theta[\cdot,\cdot,\cdot]$ is real analytic in a neighborhood of $(q_0,\phi_0,g_0)$ in 
${\mathbb{D}}_{n}^{+}({\mathbb{R}})\times\left(C^{1,\alpha}(\partial\Omega, \mathbb{R}^n) \cap {\mathcal{A}}_{\partial\Omega}^{\widetilde{Q}}\right)\times C^{0,\alpha}(\partial \Omega)$. 
\end{proof}

\begin{remark} \label{repform}
By Lemma \ref{rajacon}, Propositions \ref{propAUX} and  \ref{taxi}, we have the following representation formula for the solution $u[q,\phi,g,k]$ of problem \eqref{bvp}: 
\begin{align*}
&u[q,\phi,g,k](x) = \int_{\partial\Omega}
 S_{q,n}(x-q\phi(s))\theta[q,\phi,g](s)\tilde \sigma[q\phi](s)\,d\sigma_{s}\\
 &+\frac{ \Bigg(k-\!\int_{\partial\Omega} \int_{\partial\Omega}
 S_{q,n}(q(\phi(t)-\phi(s)))\theta[q,\phi,g](s)\tilde \sigma[q\phi](s)d\sigma_{s}\tilde{\sigma}[q\phi](t)d\sigma_{t}\Bigg)}{\int_{\partial\Omega} \!\tilde \sigma[q\phi]d\sigma}\\
&\forall x\in \overline{\mathbb{S}_{q}[q\mathbb{I}[\phi]]^-},
\end{align*}
for all $(q,\phi,g,k) \in {\mathbb{D}}_{n}^{+}({\mathbb{R}})\times\left(C^{1,\alpha}(\partial\Omega, \mathbb{R}^n) \cap {\mathcal{A}}_{\partial\Omega}^{\widetilde{Q}}\right)\times C^{0,\alpha}(\partial \Omega)\times \mathbb{R}$. 
\end{remark}

By exploiting the representation formula of Remark \ref{repform} and the analyticity result for $(q,\phi,g) \mapsto \theta[q,\phi,g]$ of Proposition \ref{taxi}, we are ready to prove our main result on the analyticity of $u[q,\phi,g,k]$ as a map of the variable $(q,\phi,g,k)$.

\begin{theorem}\label{mainthm}
Let $\alpha$, $\Omega$ be as in \eqref{Omega_def}. Let 
\[
(q_0, \phi_0,g_0,k_0) \in  {\mathbb{D}}_{n}^{+}({\mathbb{R}})\times\left(C^{1,\alpha}(\partial\Omega, \mathbb{R}^n) \cap {\mathcal{A}}_{\partial\Omega}^{\widetilde{Q}}\right)\times C^{0,\alpha}(\partial \Omega)\times \mathbb{R}.
\]
Let $U$ be a bounded open subset of $\mathbb{R}^n$ such that $\overline{U} \subseteq  \mathbb{S}_{q_0}[q_0\mathbb{I}[\phi_0]]^-$. Then there exists an open neighborhood $\mathcal{U}$ of  $(q_0, \phi_0,g_0,k_0)$  in 
$${\mathbb{D}}_{n}^{+}({\mathbb{R}})\times\left(C^{1,\alpha}(\partial\Omega, \mathbb{R}^n) \cap {\mathcal{A}}_{\partial\Omega}^{\widetilde{Q}}\right)\times C^{0,\alpha}(\partial \Omega)\times \mathbb{R}$$ such that the following statements hold.
\begin{itemize}
\item[(i)] $\overline U \subseteq  \mathbb{S}_{q}[q\mathbb{I}[\phi]]^-$ for all $(q,\phi,g,k) \in \mathcal{U}$.
\item[(ii)] Let $m \in \mathbb{N}$. Then the map from $\mathcal{U}$ to $C^m(\overline{U})$ which takes $(q,\phi,g,k)$ to the restriction 
$u[q,\phi,g,k]_{|\overline{U}}$ of $u[q,\phi,g,k]$ to $\overline{U}$ is real analytic.
\end{itemize}
\end{theorem}
\begin{proof}
We first note that, by taking $\mathcal{U}$ small enough, we can deduce the validity of (i). The validity of (ii) follows by the representation formula of Remark \ref{repform}, by Lemma \ref{rajacon}, by Proposition \ref{taxi},  by the regularity results of  \cite{LaMu13} on the analyticity of integral operators with real analytic kernels, and by standard calculus in Banach spaces.
\end{proof}

\begin{acknowledgement}
The authors are members of the `Gruppo Nazionale per l'Analisi Matematica, la Probabilit\`a e le loro Applicazioni' (GNAMPA) of the `Istituto Nazionale di Alta Matematica' (INdAM). P.L.~and P.M.~acknowledge the support of the Project BIRD191739/19 `Sensitivity analysis of partial differential equations in
the mathematical theory of electromagnetism' of the University of Padova.  
P.M.~acknowledges the support of  the grant `Challenges in Asymptotic and Shape Analysis - CASA'  of the Ca' Foscari University of Venice.  P.M.~also acknowledges the support from EU through the H2020-MSCA-RISE-2020 project EffectFact, 
Grant agreement ID: 101008140.
\end{acknowledgement}
%
%

%
%

\begin{thebibliography}{99.}%
%
%

  
  \bibitem{AmKa07}
Ammari, H., Kang H.: { Polarization and moment tensors, With applications to inverse
  problems and effective medium theory}, Springer (2007)





 \bibitem{BuPr15} 
Buoso, D., Provenzano, L.: A few shape optimization results for a biharmonic Steklov problem. J. Differential Equations {\bf 259}, no. 5, 1778–1818 (2015)



\bibitem{DaLaMu21}
Dalla Riva, M., Lanza de Cristoforis, M.,  Musolino, P.:  {Singularly Perturbed Boundary Value Problems: A Functional Analytic Approach}, Springer Nature, Cham (2021)

\bibitem{DaLuMuPu21}
Dalla Riva, M., Luzzini, P., Musolino, P.,  Pukhtaievych, R.: Dependence of effective properties upon regular perturbations.  In I. Andrianov, S. Gluzman, V. Mityushev, Editors, Mechanics and Physics of Structured Media, Elsevier, to appear.

\bibitem{De85}
Deimling, D.: {Nonlinear {F}unctional {A}nalysis}, Springer-Verlag, (1985)

\bibitem{GiTr83}
Gilbarg, D.,  Trudinger, N.S.: {Elliptic partial differential equations of second
  order}, 2nd Edition, Springer-Verlag,   (1983)

\bibitem{He82}
 Henry, D.: {Topics in nonlinear analysis}, Universidade de Brasilia, Trabalho de Matematica  \textbf{192} (1982).
 
 
 
\bibitem{LaLa04} Lamberti, P.D., Lanza~de~Cristoforis, M.:  {A real analyticity result for symmetric functions of the eigenvalues of a domain dependent Dirichlet problem for the Laplace operator}. J. Nonlinear Convex Anal. \textbf{5}(1), 19--42 (2004)


 
 
\bibitem{LaZa21} 
Lamberti, P.D., Zaccaron, M.:  Shape sensitivity analysis for electromagnetic cavities. Math. Methods Appl. Sci. {\bf 44}, no. 13, 10477--10500 (2021)
 
  \bibitem{La05}
Lanza~de~Cristoforis, M.: {A domain perturbation problem for the {P}oisson
  equation}. {Complex Var. Theory Appl.}  \textbf{50}(7-11), 851--867  (2005)


   \bibitem{La07}
Lanza~de~Cristoforis, M.: {Perturbation problems in potential theory, a
  functional analytic approach}. {J. Appl. Funct. Anal.} \textbf{2}(3), 197--222 (2007)
  


\bibitem{LaMu13}
Lanza~de~Cristoforis, M., Musolino, P.: {A real analyticity result for a nonlinear integral operator}. J. Integral Equations Appl. \textbf{25}(1), 21--46  (2013)

 
 
 
 \bibitem{LaRo04}
Lanza~de~Cristoforis, M., Rossi, L.: { 
Real analytic dependence of simple and double 
layer potentials upon perturbation 
of the support and of the density}. {J. Integral Equations 
Appl.}  {\bf 16}, 137--174 (2004)


\bibitem{LuMu20}
Luzzini, P., Musolino, P.: Perturbation analysis of the effective conductivity of a periodic composite. Netw. Heterog. Media, {\bf 15}, no. 4, 581--603 (2020)



\bibitem{LuMu22} 
Luzzini, P., Musolino, P.. Domain perturbation for the solution of a periodic Dirichlet problem. In  Cerejeiras, P., Reissig, M., Sabadini, I.,  Toft, J. Editors: Current Trends in Analysis, its Applications and Computation, Proceedings of the 12th ISAAC congress (Aveiro, 2019), Research Perspectives, Birkh\"auser, to appear.



\bibitem{LuMuPu19}
Luzzini, P., Musolino, P., Pukhtaievych, R.: Shape analysis of the longitudinal flow along a periodic array of cylinders. J. Math. Anal. Appl., {\bf 477}, no. 2, 1369--1395 (2019)






  
  
  
      \bibitem{LuMuPu20}
Luzzini, P., Musolino, P., Pukhtaievych, R.: {Real analyticity of periodic layer potentials upon perturbation of the periodicity parameters and of the support}. In  \emph{Proceedings of the 12th ISAAC congress (Aveiro, 2019)}, Research Perspectives, Birkh\"auser, to appear.
  


 
 \bibitem{SoZo92}
  Sokolowski, J., Zol\'esio, J.P.:  {Introduction to Shape Optimization. Shape Sensitivity Analysis}, Springer-Verlag (1992).

\end{thebibliography}
%

\end{document}